\documentclass[11pt]{amsart}
\usepackage[centering,margin=2.5cm,a4paper]{geometry}
\usepackage{amssymb}
\usepackage{color}
\usepackage{soul}
\theoremstyle{plain}
\newtheorem{theorem}{Theorem}[section]
\newtheorem{lemma}[theorem]{Lemma}
\newtheorem{corollary}[theorem]{Corollary}
\newtheorem{proposition}[theorem]{Proposition}

\theoremstyle{definition}

\theoremstyle{remark}
\newtheorem{remark}[theorem]{Remark}

\renewcommand{\le}{\leqslant}
\renewcommand{\leq}{\leqslant}
\renewcommand{\ge}{\geqslant}
\renewcommand{\geq}{\geqslant}
\renewcommand{\hat}{\widehat}
\DeclareMathOperator{\conv}{\mathrm{Conv}}
\DeclareMathOperator{\aff}{\mathrm{aff}}
\DeclareMathOperator{\measure}{\mathrm{measure}}

\newcommand{\R}{\mathbb{R}}

\newcommand{\flo}[1]{\left\lfloor{#1}\right\rfloor}
\newcommand{\cei}[1]{\left\lceil{#1}\right\rceil}

\title[Complete Kneser Transversals]{Complete Kneser Transversals}
\thanks{Supported by the ECOS-Nord project M13M01, by CONACyT project 166306, by CONACyT grant 277462, by PAPIIT-UNAM project IN112614 and by Frenkel Foundation}
\thanks{$^{\star}$ Corresponding Author: jonathan.chappelon@umontpellier.fr}
\author[\tiny J.~Chappelon]{J.~Chappelon$^{\, \star}$}
\address{Institut Montpelli\'{e}rain Alexander Grothendieck, Universit\'{e} de Montpellier, Case Courrier 051, Place Eug\`{e}ne Bataillon, 34095 Montpellier Cedex 05, France}
\email{jonathan.chappelon@umontpellier.fr}
\author[L.~Mart\'{i}nez-Sandoval]{L.~Mart\'{i}nez-Sandoval}
\address{Instituto de Matem\'{a}ticas, Universidad Nacional Aut\'{o}noma de M\'{e}xico, Ciudad Universitaria, M\'{e}xico D.F., 04510, Mexico}
\email{leomtz@im.unam.mx}
\author[L.~Montejano]{L.~Montejano}
\address{Instituto de Matem\'{a}ticas, Universidad Nacional Aut\'{o}noma de M\'{e}xico, Ciudad Universitaria, M\'{e}xico D.F., 04510, Mexico}
\email{luis@math.unam.mx}
\author[L.P.~Montejano]{L.P.~Montejano}
\address{Institut Montpelli\'{e}rain Alexander Grothendieck, Universit\'{e} de Montpellier, Case Courrier 051, Place Eug\`{e}ne Bataillon, 34095 Montpellier Cedex 05, France}
\email{lpmontejano@gmail.com}
\author[J.L.~Ram\'{i}rez Alfons\'{i}n]{J.L.~Ram\'{i}rez Alfons\'{i}n}
\address{Institut Montpelli\'{e}rain Alexander Grothendieck, Universit\'{e} de Montpellier, Case Courrier 051, Place Eug\`{e}ne Bataillon, 34095 Montpellier Cedex 05, France}
\email{jorge.ramirez-alfonsin@umontpellier.fr}
\subjclass[2010]{52A35, 05C15, 52C40, 52B40, 52B99}
\keywords{transversals, Kneser hypergraphs, oriented matroids, cyclic polytope}
\date{August 10, 2016}
\begin{document}
\begin{abstract}
Let $k,d,\lambda\geqslant1$ be integers with $d\geqslant\lambda $. Let $m(k,d,\lambda)$ be the maximum positive integer $n$ such that every set of $n$ points (not necessarily in general position) in $\mathbb{R}^{d}$ has the property that the convex hulls of all $k$-sets have a common transversal $(d-\lambda)$-plane. It turns out that  $m(k, d,\lambda)$ is strongly connected with other interesting problems, for instance, the chromatic number of Kneser hypergraphs and a discrete version of Rado's centerpoint theorem. In the same spirit, we introduce a natural discrete version $m^*$ of $m$ by considering the existence of \emph{complete Kneser transversals}. We study the relation between them and give a number of lower and upper bounds of $m^*$ as well as the exact value in some cases. The main ingredient for the proofs are Radon's partition theorem as well as oriented matroids tools. By studying the alternating oriented matroid we obtain the asymptotic behavior of the function $m^*$ for the family of cyclic polytopes.
\end{abstract}
\maketitle
\section{Introduction}
Let $k,d,\lambda \geq 1$ be integers with both $d, k \geq \lambda $. Consider the function $m(k,d,\lambda )$, defined to be the maximum positive integer $n$ such that every set of $n$ points (not necessarily in general position) in $\R^{d}$ has the property that the convex hulls of all $k$-sets have a common transversal $(d-\lambda)$-plane.
\smallskip

In \cite{ABMR}, the following inequalities were obtained
\begin{equation}\label{ineq1}
d-\lambda +k+\cei{\frac{k}{\lambda }} -1\leq m(k,d,\lambda ) < d+2(k-\lambda)+1  .
\end{equation}
An interesting feature of the value of $m(k,d,\lambda )$ is its strong connection with the chromatic number of Kneser hypergraphs  \cite{kneser, lovasz} as well as with the Rado's centerpoint theorem \cite{Rado}. Indeed, for the former it is proved in \cite{ABMR} that
\begin{equation*}
\text{if } \: m(k,d,\lambda) < n, \: \text{ then } \:   d-\lambda +1 < \chi \left(KG^{\lambda +1}(n,k)\right).
\end{equation*}
For the latter, recall that the well-known Rado's centerpoint theorem \cite{Rado} states that if $X$ is a bounded measurable set in $\R^d$ then there exists a point $x\in\R^d$ such that
$$
\measure\left(P\cap X\right)\ge \frac{\measure\left(X\right)}{d+1}
$$
for each half-space $P$ that contains $x$ (see also \cite{Neumann} for the case $d=2$).

Independently Bukh and Matousek \cite[Section 6]{BM} and Arocha, Bracho, Montejano and Ram\'{i}rez-Alfons\'{i}n in \cite{ABMR} consider the following generalization of a discrete version of Rado's centerpoint theorem. Let $n,d,\lambda \geq 1$ be integers with $d  \geq \lambda $ and let

\smallskip
$\tau(n,d,\lambda )\overset{\mathrm{def}}{=}$ the maximum positive integer $\tau$
such that for any collection $X$ of $n$ points in $\R^{d}$, there is a  $(d-\lambda )$-plane $L_X$ such that any closed half-space $H$ through $L_X$ contains at least $\tau$ points.
\smallskip

By the hyperplane separation theorem we have that $n-\tau(n,d,\lambda)+1$ is equal to the minimum positive integer $k$ such that for any collection $X$ of $n$ points in $\R^{d}$  there is a common transversal $(d-\lambda)$-plane to the convex hulls of all $k$-sets, which is essentially $m(k,d,\lambda)$. Therefore, any improvement to the lower or upper bounds for $m(k,d,\lambda)$ will give important insight on the above interesting problem.

The purpose of this paper is to introduce and study a discrete version of the function $m(k,d,\lambda)$ which perhaps will allow us to improve lower or upper bounds for $m(k, d,\lambda)$. Let $k,d,\lambda \geq 1$ be integers with $k,d\geq \lambda $ and let $X\subset\R^{d}$ be a finite set. We call $L$ a \emph{Kneser transversal} of $X$ if it is a $(d-\lambda)$-plane transversal to the convex hulls of all $k$-sets of $X$. If in addition $L$ contains $(d-\lambda)+1$ points of $X$, then $L$ is a \emph{complete Kneser $(d-\lambda)$-transversal}.  Let us define

\smallskip
$m^*(k,d,\lambda )\overset{\mathrm{def}}{=}$ the maximum positive integer $n$ such that every set of $n$ points (not necessarily in general position) in $\R^{d}$ has a complete Kneser $(d-\lambda )$-transversal to the convex hulls of its $k$-sets.
\smallskip

This is a natural discrete version of the original function $m$. Indeed, consider a set of points $X$ in $\R^d$. The existence of an arbitrary $(d-\lambda)$-plane transversal to the convex hull of the $k$-sets of $X$ is not an invariant of the order type. For example, if $d=2$ and $X$ is the vertex set of a regular hexagon then the center is a $0$-plane transversal to the convex hull of the $4$-sets. But by suitably perturbing these $6$ points slightly we lose this property.
On the other hand, the existence of a complete Kneser $(d-\lambda)$-transversal to the convex hull of the $k$-sets is an invariant of the order type (see Propositions~\ref{propMinTr} and \ref{propRadTr}). This allows us to study $m^*$ using oriented matroid theory. Since the parameter $m^*$ requires additional conditions on the transversals, we clearly have
$$
m^*(k,d,\lambda) \le m(k,d,\lambda).
$$
The case $k=\lambda$ is easy to deal with (see Proposition~\ref{propEasy}), so from here on we will assume that $k\geq \lambda+1$. It turns out that the function $m^*$ has two different behaviors. The arguments for the case $\lambda-1\geq\cei{\frac{d}{2}}$, are usually simpler than those for the case $\lambda-1<\cei{\frac{d}{2}}$. For this reason, we define 
$$
\alpha(d,\lambda)=\frac{\lambda-1}{\left\lceil \frac{d}{2}\right\rceil}
$$
and we call $\alpha\geq 1$ the trivial range and $\alpha<1$ the non-trivial range.
This paper is devoted to investigate $m^*$. We present the exact value of $m^*$ in the trivial range and give bounds and some exact values of $m^*$ in the non-trivial range. The paper is organized as follows. In Section~\ref{radon} we provide tools from convex geometry to detect complete Kneser transversals using Radon partitions. We find the following lower bound for $m^*$.

\begin{theorem}\label{lowerbound}
In the non-trivial range, when $\alpha(d,\lambda)<1$, we have that
$$
 (d-\lambda+1)+k\le m^*(k,d,\lambda).
$$
\end{theorem}

In order to give an upper bound for $m^*(k,d,\lambda)$, we study a specific family of sets of points. We do this in Section~\ref{cyclic}, where we review cyclic polytopes and alternating oriented matroids. We introduce the following function.

\smallskip
$\zeta(k,d,\lambda)\overset{\mathrm{def}}{=}$ the maximum number of vertices that the cyclic polytope in $\R^d$ can have, so that it has a complete Kneser $(d-\lambda)$-transversal to the convex hulls of its $k$-sets of vertices.
\smallskip

We clearly have
$$
m^*(k,d,\lambda)\leq \zeta(k,d,\lambda).
$$
This helps us to establish the following theorem that completely solves the problem in the trivial range.

\begin{theorem}\label{cyceasy}
When $\alpha(d,\lambda)\geq 1$, we have that
$$
m^*(k, d,\lambda)= \zeta(k,d,\lambda)=d-\lambda + k.
$$
In particular, when $\alpha(d,\lambda)\geq 1$, the vertex set of the cyclic polytope with at least $(d-\lambda+1)+k$ points does not have a complete Kneser $(d-\lambda)$-transversal to the convex hulls of its $k$-sets.
\end{theorem}

When $\alpha(d,\lambda)<1$, we present upper and lower bounds for $\zeta(k,d,\lambda)$ in the non-trivial range (see Theorem \ref{thCyclic}).  This allows us to obtain an upper bound for $m^*$ in the non-trivial range.
We end Section~\ref{cyclic} by showing that these bounds are asymptotically correct in terms of $k$.

\begin{theorem}\label{asymp1}
In the non trivial-range we have that
$$
\lim_{k\to\infty}\frac{\zeta(k,d,\lambda)}{k}=2-\alpha(d,\lambda).
$$
\end{theorem}

As an easy consequence, we have
$$
\lim_{d\to\infty}\lim_{k\to\infty}\frac{\zeta(k,d,\lambda)}{k}=2.
$$
Also, by combining Theorem \ref{asymp1} with a result of Bukh and Matousek \cite{BM}, we verify that if $\alpha(d,\lambda)< 1$, then $m(k,d,2)$ is not necessarily equal to $m^*(k,d,2)$ (see Remark \ref{rem:mm}).
\medskip
Finally, in Section~\ref{exact} we present some exact values of $m^*$. Among other, we prove the following two results.

\begin{theorem}\label{exactvaluesm^*}
Let $\alpha(d,\lambda)<1$ and $\frac{d}{2}\leq k < \frac{2}{1-\alpha(d,\lambda)}$. Then,
$$
m^*(k,d,\lambda)=(d-\lambda+1)+k.
$$
\end{theorem}

\begin{theorem}\label{m^*(k,d,1}
$m^*(k,d,1)=d+2(k-1)$.
\end{theorem}

\section{Kneser transversals from Radon partitions}\label{radon}

Let $d$ be a positive integer. Consider $d+2$ points $v_1$, $v_2$, $\ldots$, $v_{d+2}$ in general position in $\R^d$. Radon's theorem states that there exists a unique partition $\{1,2,\ldots,d+2\}=A\cup B$ such that
$$
\conv\left(\bigcup_{i\in A}{v_i}\right)\cap\conv\left(\bigcup_{i\in B} v_i\right)\neq \emptyset.
$$
Moreover, Radon's theorem states that this intersection is a unique point in the interior of each convex hull.

The following proposition is a generalization of the well-known Carath\'{e}odory's theorem that states that if a point $p$  lies in the convex hull of a set $S$ in  $\R^d$, then there is a subset $S'$ of $S$ consisting of at most $d+1$ points such that $p$ lies in the convex hull of $S′$. It is not difficult to prove that the set $S'$ has exactly $d+1$ points if the set $S$ is in general position in  $\R^d$.

\begin{proposition}\label{propMinTr}
Let $d$ and $\lambda$ be positive integers with $d\ge\lambda$ and let $S$ and $T$ be two disjoint sets of points in general position in $\R^d$ such that $|S|\ge\lambda+1$ and $|T|=d-\lambda+1$. Then the following two statements are equivalent:	
\begin{itemize}
\item
$\conv(S)\cap \aff(T)\neq 0$,
\item
$\conv(S')\cap \aff(T)\neq 0$ for a subset $S'\subseteq S$ such that $|S'|=\lambda+1$.
\end{itemize}
\end{proposition}

\begin{proof}
The second statement clearly implies the first one. Now let $p\in\conv(S)\cap \aff(T)$ and let $U$ be the affine subspace through $p$ that is perpendicular to $\aff(T)$. Consider $S_U$ the projection of $S$ to $U$. We have that $p\in\conv\left(S_U\right)$ and that $U$ has dimension $\lambda$. Therefore, by Carath\'{e}odory's theorem, there is a $(\lambda+1)$-set of $S_U$ such that $p$ lies in its convex hull. This corresponds to a subset $S'$ of $S$ with $\lambda+1$ elements such that $\conv(S')\cap \aff(T)\neq 0$, as desired.
\end{proof}

The following result will be very useful for the rest of the paper.

\begin{proposition}\label{propRadTr}
Let $S$ and $T$ be two disjoint and non-empty sets of points in $\R^d$ such that $|S|+|T|=d+2$ and $S\cup T$ is in general position. Then $\conv(S)\cap\aff(T)\neq \emptyset$ if and only if all the points of $S$ are in the same set in the Radon partition of $S\cup T$.
\end{proposition}

\begin{proof}
Let $S\cup T=\{v_1,v_2,\ldots, v_{d+2}\}$ and let $A\cup B$ be the partition of $\{1,\ldots,d+2\}$ that yields the Radon partition for $S\cup T$. This means that there exist positive real numbers $\alpha_1$, $\ldots$, $\alpha_{d+2}$ such that
$$
\begin{array}{l}
\displaystyle\sum_{i\in A} \alpha_i v_i = \sum_{i\in B} \alpha_i v_i, \\ \ \\
\displaystyle\sum_{i\in A} \alpha_i = \sum_{i\in B} \alpha_i = 1. \\
\end{array}
$$
First suppose that all the elements from $S$ are in the same part in the Radon partition of $S\cup T$. Without loss of generality, this means that if $v_i\in S$, then $i\in A$. We may therefore write:
$$
\sum_{i\in A, v_i\in S} \alpha_i v_i = \sum_{i\in B} \alpha_i v_i - \sum_{i\in A, v_i\in T} \alpha_i v_i.
$$
Dividing the both sides of the equality by $\sum_{i\in A, v_i\in S} \alpha_i$, we get on the LHS a convex linear combination of elements in $S$ and on the RHS an affine linear combination of elements in $T$. This shows that $\conv(S)\cap\aff(T)\neq \emptyset$.

Now suppose that $\conv(S)\cap\aff(T)\neq \emptyset$. This means that there exist real numbers $\beta_i$ such that
$$
\begin{array}{l}
\beta_i\geq 0 \text{ when } v_i\in S, \\ \ \\
\displaystyle\sum_{v_i\in S} \beta_i v_i = \sum_{v_i\in T} \beta_i v_i, \\ \ \\
\displaystyle\sum_{v_i\in S} \beta_i = \sum_{v_i\in T} \beta_i = 1. \\
\end{array}
$$
We may rearrange the sum as
$$
\sum_{v_i\in S} \beta_i v_i + \sum_{v_i\in T, \beta_i<0} (-\beta_i) v_i = \sum_{v_i\in T, \beta_i\geq 0} \beta_i v_i\\
$$
and dividing both sides of the equality by $\sum_{v_i\in T, \beta_i\geq 0} \beta_i$ we get a convex linear combination on both sides. This induces a Radon partition of $S\cup T$. Since the points are in general position, this must be the same partition as the one induced by $A\cup B$. Therefore $\alpha_i=\beta_i$ and all the $\alpha_i$'s corresponding to points in $S$ are in the same part.
\end{proof}

\begin{remark}\label{remNoGen}
In the case in which $S\cup T$ is not in general position then we still have a Radon partition, but it might not be unique. If in one of those Radon partitions we have that all the points from $S$ belong to the same set, then the proof above shows that we also conclude that $\text{Conv}(S)\cap \text{aff}(T)\neq \emptyset$.
\end{remark}

As a consequence of the above proposition and remark we get the following lemma. 

\begin{lemma}\label{convaff3}
Let $X$ be any set of $d+2$ distinct points in $\R^d$ and let $\flo{\frac{d+2}{2}}\le t\le d+1$. Then $X$ can be partitioned into two disjoint sets $S$ and $T$ such that $|T|=t$ and $\conv(S)\cap\aff(T)\neq \emptyset$.
\end{lemma}

\begin{proof}
By Radon's theorem the set $X$ can be partitioned into two disjoint sets, $A$ and $B$, whose convex hulls intersect. We may suppose that $|B|\le\flo{\frac{d+2}{2}}\le t$, since $|A|+|B|=d+2$. If $|B|=t$, as $\conv(A)\cap\conv(B)\neq\emptyset$, we have already finished the proof. If $|B|<t$, let $Y\subseteq A$ be such that $|B|+|Y|=t$. Define $T=B\cup Y$ and $S=A\setminus Y$, hence by Proposition~\ref{propRadTr} we conclude that $\conv(S)\cap\aff(T)\neq\emptyset$.
\end{proof}

\subsection{A lower bound for $m^*(k,d,\lambda)$ in the non-trivial range $\alpha(d,\lambda)<1$}

\begin{proof}[Proof of Theorem~\ref{lowerbound}]
Let $X$ be a collection of $(d-\lambda+1)+k$ points in $\R^d$. Since $k\geq \lambda+1$, then $|X|\geq d+2$. Let $Y$ be a $(d+2)$-subset of $X$.

Since we are in the case $\alpha(d,\lambda)<1$, we have that $\flo{\frac{d+2}{2}}\leq d-\lambda+1\leq d+1$. Therefore, by Lemma~\ref{convaff3} we can give a partition of $Y$ into two disjoint sets $S$ and $T$ such that $|T|=d-\lambda+1$ and $\conv(S)\cap\aff(T)\neq\emptyset$.

We claim that $\text{aff}(T)$ is a Kneser transversal for $X$. If a $T$ intersects a $k$-set, then $\text{aff}(T)$ is clearly a transversal to its convex hull. Since $X$ has $(d-\lambda+1)+k$ points, there is exactly one $k$-subset that does not intersect $T$: the complement of $T$ in $X$. But this $k$-set contains $S$, for which we know $\conv(S)\cap\aff(T)\neq\emptyset$. This shows that $\text{aff}(T)$ is a transversal to the convex hull of all $k$-sets.
\end{proof}

\section{Matroids and cyclic polytopes}\label{cyclic}

The \emph{moment curve} in $\R^d$ is defined parametrically as the map $\gamma: \R\to\R^d$, $t\mapsto (t,t^2,\ldots, t^d)$. A \emph{cyclic polytope} is the convex hull of a finite set of points on the moment curve. In this section we study the function $m^*$ for sets of vertices of cyclic polytopes. For basic notions of oriented matroids we refer the reader to \cite{BL}. The oriented matroids associated to cyclic polytopes on $n$ vertices of dimension $d$ are called \emph{alternating oriented matroids} and they are denoted by $A(r,n)$ with $r=d+1$. A well-known fact in oriented matroid theory is that the circuits of oriented matroid theory arising from a configuration of points can be interpreted as minimal Radon-partitions induced by the signs of the elements. For example, if we have the set of points $V=\{v_1,v_2,v_3,v_4,v_5,v_6\}$ in $\R^3$ and if $C=\{v_1,v_2,v_4,v_5,v_6\}$ is a signed circuit with $++--+$, this means that the sets $A=\{v_1,v_2,v_6\}$ and $B=\{v_4,v_5\}$ form a Radon partition, that is $\conv(A)\cap\conv(B)\neq\emptyset$.

Suppose that the ground set of $A(r,n)$ is $[n]$ and let $C$ be one of its circuits. A well-known fact \cite[Section 9.4]{BL} is that
\begin{align}
&\text{$|C|=r+1$ and if its elements are increasingly ordered, then they}\nonumber \\ 
&\text{are alternatively signed $+-+-\cdots$} \label{thRadCyc}
\end{align}
Therefore, minimal Radon partitions of cyclic polytopes are well understood.

Recall that $\zeta(k,d,\lambda)$ was defined as the maximum number of vertices that the cyclic polytope in $\R^d$ can have, so that it has a complete Kneser $(d-\lambda)$-transversal to the convex hulls of all its $k$-subsets of vertices.

We clearly have,
$$
m^*(k, d,\lambda)\leq \zeta(k,d,\lambda).
$$
We will give upper and lower bounds for $\zeta(k,d,\lambda)$. First we deal with some easy special cases. If $\lambda=0$, then any $d-0$ transversal is the whole space, and then we can have as many points as we want. Also, in the trivial range Theorem~\ref{cyceasy} states that the precise value of $\zeta(k,d,\lambda)$ and $m^*(k,d,\lambda)$ is $d-\lambda + k$. We now prove this.

\begin{proof}[Proof of Theorem~\ref{cyceasy}]
Clearly $d-\lambda+k\le m^*(k,d,\lambda)$ since every $(d-\lambda+1)$-set intersects all the $k$-sets for any set of $d-\lambda+k$ points in $\R^{d}$. Let $S$ be the cyclic polytope with $d-\lambda+1+k$ points in $\R^d$. Let $T\subseteq S$ be any set with $d-\lambda+1$ points, $K$ a subset of $S\setminus T$ with $k$ points and $K'$ a subset of $K$ with $\lambda+1$ points. By (\ref{thRadCyc}) the Radon partition of $T\cup K'$ can have at most $\cei{\frac{d+2}{2}}$ elements with the same sign. Since by hypothesis $|K'| = \lambda+1 >\cei{\frac{d+2}{2}}$, then $K'$ has at least two elements with different signs and hence by Proposition~\ref{propRadTr} we have $\conv(K')\cap\aff(T)=\emptyset$. Therefore $\conv(K)\cap\aff(T)=\emptyset$ by Proposition~\ref{propMinTr}.
\end{proof}

Let us define $\beta(\lambda,j)=\frac{j+\lambda-1}{2}$ for each integer $j$ such that $j+\lambda$ is an odd number. Let
$$
\begin{array}{l}
\displaystyle z(k,d,\lambda ) \overset{\mathrm{def}}{=} (d-\lambda+1) + \max_{\substack{j\in \{\lambda+1,\ldots,d-\lambda+2\}\\ j+\lambda \text{ is odd}}} \left(\flo{\frac{k-1}{\beta(\lambda,j)}}\cdot j + (k-1)_{\text{mod} \beta(\lambda,j)}\right) \\
\displaystyle Z(k,d,\lambda ) \overset{\mathrm{def}}{=} (d-\lambda+1) + \flo{(2-\alpha(d,\lambda))(k-1)} \\
\end{array}
$$
The rest of this section is devoted to prove the following bounds.

\begin{theorem}\label{thCyclic}
In the non-trivial range, when $\alpha(d,\lambda)< 1$,
$$
z(k,d,\lambda) \leq \zeta (k,d,\lambda) \leq Z(k,d,\lambda).
$$
\end{theorem}

\subsection{Some combinatorial tools}\label{secTools}

In this section we develop some combinatorial tools that will allow us to prove Theorem~\ref{thCyclic}. Let $n$ be a positive integer. We denote by $[n]$ the set $\{1,2,\ldots,n\}$, as usual.

For a set $S$ of integers $x_1<x_2<\ldots<x_r$ we will denote by $OD(S)$ the number of odd integers in the set of differences
$$
\{x_2-x_1,x_3-x_2,\ldots,x_r-x_{r-1}\}.
$$
In other words, a set $S$ can be split in $OD(S)+1$ \emph{parity blocks}, that is, maximal blocks of adjacent elements with the same parity. For example, the set $S=\{1,4,5,7,8,10,12\}$ satisfies $OD(S)=3$ and its parity blocks are $\{1\}$, $\{4\}$, $\{5,7\}$ and $\{8,10,12\}$. It is easy to verify that if $S$ is a subset of $T$, then $OD(S)\leq OD(T)$.

Fix a non-negative integer $\ell$ and a positive integer $k$. In order to prove the lower bound for Theorem~\ref{thCyclic}, we are interested in finding the largest value of $n$ such that each $k$-set of $[n]$ contains a subset $S$ such that $OD(S)\geq \ell$. Let $D(k,\ell)$ denote this maximum.

\begin{proposition}\label{propAlter}
For a non-negative integer $\ell$ and a positive integer $k$ we have
$$
D(k,\ell) = \left\{\begin{array}{ll}
\displaystyle\infty & \text{if } \ell=0 \\
\displaystyle 2k-\ell-1 & \text{if } k\geq \ell\geq 1 \\
\displaystyle k-1 & \text{if } \ell\geq k.
\end{array}\right.
$$
\end{proposition}

\begin{proof}
We will first establish the formula for the cases $\ell=0$ and $\ell\geq k$. For any element $j$ of $[n]$ we have $OD(\{j\})=0$ therefore, if $\ell=0$ then $n$ can be as large as we want.
	
Now suppose that $\ell\geq k$. If $n\leq k-1$, then there are no $k$-sets. If $n\geq k$, then there is at least one $k$-set $K$, but any subset $S$ of $K$ set satisfies $OD(S)\leq OD(K)\leq k-1$.
	
We are left with the case $k\geq\ell\geq 1$. Suppose that $n\leq 2k-\ell-1$ and consider a $k$-set of $[n]$ with ordered elements $x_1<x_2<\ldots<x_k$. If it does not have any subset $S$ with $OD(S)\geq \ell$, then at most $\ell-1$ of the numbers in $\{x_2-x_1,x_3-x_2,\ldots,x_k-x_{k-1}\}$ are odd, and thus at least $(k-1)-(\ell-1)=k-\ell$ of them are even giving
$$
x_k-x_1=\sum_{j=1}^{k-1}\left(x_{j+1}-x_j\right) \geq \ell-1 + 2 (k-\ell) = 2k-\ell-1.
$$
This implies $n\geq x_k\geq 2k-\ell-1 +x_1\geq 2k-\ell$, a contradiction.
	
Now, if $n\geq 2k-\ell$, then the set
$$
K=\{1,2,3,\ldots,\ell-1\}\cup\{\ell,\ell+2,\ell+4,\ldots,\ell+2(k-\ell)\}
$$
is a $k$-set of $[n]$. Notice that $OD(K)=\ell-1$ and therefore each subset $S$ satisfies $OD(S)\leq \ell-1$.
\end{proof}

We now develop further combinatorial tools needed for the proof of the upper bound for Theorem~\ref{thCyclic}. Let $d$ and $\lambda$ be two positive integers. We will construct a special family of subsets of $[d-\lambda+2]$ as follows.

First, we consider the case in which $d$ is odd. We define
$$
I(d,\lambda): =\{1,2,\ldots,\lambda-1\}\cup\{\lambda,\lambda+2,\lambda+4,\ldots,d-\lambda+1\}.
$$
The parity of $d$ ensures that the last term is correct. Notice that this set has
$$
(\lambda-1)+\frac{(d-\lambda+2)-(\lambda-1)}{2}=\frac{d+1}{2}
$$
elements.

Now, let $\sigma$ be the cyclic permutation in $[d-\lambda+2]$ defined as $\sigma(i)=i+1$ for $i\in [d-\lambda+1]$ and $\sigma(d-\lambda+2)=1$. For each $j\in[d-\lambda+2]$ we define the set
$$
I(d,\lambda,j) := \sigma^{j-1}(I(d,\lambda)).
$$
Finally, we define the sets $I(d,\lambda,j)$ for even values of $d$. In this case we will only define the sets up to $j=d-\lambda+1$ as follows:
$$
I(d,\lambda,j) := \left\{\begin{array}{ll}
I(d-1,\lambda,j)\cup\{d-\lambda+2\} &  \text{if $1$ is in $I(d-1,\lambda,j)$,} \\
I(d-1,\lambda,j) &  \text{if not}.
\end{array}\right.
$$
Notice that for fixed $\lambda$ and $d$, we have defined a total of $2\cei{\frac{d}{2}}-\lambda+1$ sets regardless of the parity of $d$. We now prove that the elements from $[d-\lambda+2]$ are well distributed in these sets and that the sets have few parity blocks.

\begin{proposition}\label{propInter}
Fix two positive integers $d$ and $\lambda$. Define the sets $I(d,\lambda,j)$ as above. Then
\begin{enumerate}
\item Each number from $[d-\lambda+2]$ appears in exactly $\cei{\frac{d}{2}}$ of the sets $I(d,\lambda,j)$.
\item For each $j$ we have $OD(I(d,\lambda,j))\leq \lambda - 1$.
\end{enumerate}
\end{proposition}

\begin{proof}
\begin{enumerate}
\item[]
\item
If $d$ is odd, this follows from the fact that $\sigma$ is a permutation of order $d-\lambda+2$ and that $I$ has $\frac{d+1}{2}=\cei{\frac{d}{2}}$ elements. If $d$ is even, by the previous argument all the elements of $[d-\lambda+1]$ appear in $\frac{(d-1)+1}{2}=\cei{\frac{d}{2}}$ of the subsets. The claim also follows for the element $d+\lambda+2$ because it appears in as many sets as $1$.
\item
We will prove that
$$
OD(I(d,\lambda,j))= \lambda-1
$$
holds for even values of $d$. This will be enough because then $I(d-1,\lambda, j)\subseteq I(d,\lambda,j)$, and therefore $OD(I(d-1,\lambda,j)) \leq OD(I(d,\lambda,j))= \lambda-1$.

If $j=1$, then the set is
$$
I(d,\lambda,1): =\{1,2,\ldots,\lambda-1\}\cup\{\lambda,\lambda+2,\lambda+4,\ldots,d-\lambda, d-\lambda+2\}.
$$
which has exactly $\lambda-1$ parity changes. We will now proceed by induction and compare $I(d,\lambda,j)$ and $I(d,\lambda,j+1)$. All the parity changes in $[d-\lambda]$ are shifted to parity changes in the interval $[2,\ldots,d-\lambda+1]$, therefore we can only increase or decrease parity changes at the endpoints.
		
At the left endpoint a shift can only increase parity changes. This happens if and only if after the shift we obtain $1$ and $2$, if and only if before the shift we had $1$ and $d-\lambda+1$. Therefore, before and after the change we have the element $d-\lambda+2$. If the number to the left of $d-\lambda+1$ was $d-\lambda$, then we lose the parity change from $d-\lambda$ to $d-\lambda+1$ If the number to the left of $d-\lambda+1$ was $d-\lambda-1$, then we lose the parity change from $d-\lambda+1$ to $d-\lambda+2$. Either way, the total changes of parity remain constant.
		
If we obtain a parity change using the right endpoint it was because we had both $d-\lambda$ and $d-\lambda+1$ but not $d-\lambda+2$. But in this case the win is compensated by the loss of the parity change from $d-\lambda$ to $d-\lambda+1$.
		
Finally, if we lose a parity change at the right endpoint it is because we had the elements $d-\lambda+1$ and $d-\lambda+2$, but this gets compensated by an additional parity change at the left endpoint.
\end{enumerate}
\end{proof}

\subsection{Upper bound of Theorem~\ref{thCyclic}}\label{secUp}

In this section we prove that $\zeta(k,d,\lambda) \leq Z(k,d,\lambda)$. Let $\gamma : \R\to\R^d$ be the moment curve and let $t_1<\ldots<t_n$ be some real positive numbers. Consider the cyclic polytope with vertex set $V=\{v_1,v_2,\ldots,v_n\}$ where $v_i=\gamma(t_i)$. We want to show that if there exists a complete Kneser $(d-\lambda)$-transversal to all the convex hulls of $V$, then $n\leq Z(k,d,\lambda)$.
	
Consider some indices $1\leq i_1<i_2<\ldots<i_{d-\lambda+1}\leq n$. For each $j\in[d-\lambda+1]$ define $w_j$ as $v_{i_j}$. Furthermore, define
$$
\begin{array}{rcl}
\displaystyle A_1 & := & \left\{v_1,\ldots,v_{i_1-1}\right\}, \\[1.25ex]
\displaystyle A_j & := & \left\{v_i | i\in \{i_{j-1}+1,\ldots,i_j-1\}\right\}\text{ for }j\in\{2,3,\ldots,d-\lambda+1\}, \\[1.25ex]
\displaystyle A_{d-\lambda+2} & := & \left\{v_{i_{d-\lambda+1}+1},\ldots,v_n\right\}, \\[1.25ex]
\mathcal{A} & := & \left\{A_1,A_2,\ldots,A_{d-\lambda+2}\right\}, \\[1.25ex]
\displaystyle T & := & \left\{w_1,\ldots,w_{d-\lambda+1}\right\}. \\
\end{array}
$$
We can get a better intuition of what is going on by concatenating the $A's$ and the $w's$ as follows:
$$
V = A_1 \, w_1 \, A_2 \, w_2 \, A_3 \,\ldots \, A_{d-\lambda+1} \, w_{d-\lambda+1} \, A_{d-\lambda+2}.
$$
We want a tool to test whether $\aff(T)$ could be a transversal to all the convex hulls of the $k$-sets. The tool is precisely the parity changes introduced in Subsection~\ref{secTools}. The following proposition makes the key connection.

\begin{proposition}\label{propIneq}
Using the notation above, if $\aff(T)$ is a transversal to all the convex hulls of the $k$-sets then for every subset $S$ of $[d-\lambda+2]$ such that $OD(S)\leq\lambda-1$ we have
$$
\sum_{\alpha \in S} |A_\alpha| \leq k-1.
$$
\end{proposition}

\begin{proof}
Suppose that $S$ is a subset of $[d-\lambda+2]$ such that $OD(S)\leq\lambda - 1$ and
$$
\sum_{\alpha \in S} |A_\alpha| \geq k.
$$
Then we can choose a $k$-set $K\subseteq \bigcup_{\alpha \in S} A_\alpha$. According to Proposition~\ref{propMinTr}, checking whether $\aff(T)\cap\conv(K)\neq\emptyset$ is the same as checking that $\aff(T)\cap\conv(K')\neq\emptyset$ for every set $K'\subseteq K$ such that $|K'| = d-(d-\lambda)+1=\lambda+1$.

Suppose $K'=\{v_{j_1},\ldots,v_{j_{\lambda+1}}\}$ for some indices $j_1<\ldots<j_{\lambda+1}$. Each element of $K'$ is in some element of $\mathcal{A}$. For each $r\in [\lambda+1]$, let $\hat{j_r}$ be the index such that $v_{j_r}\in A_{\hat{j_r}}$. Notice that
\begin{itemize}
\item
$\hat{j_1}\leq \hat{j_2}\leq \ldots \leq \hat{j_{\lambda+1}}$,
\item
$\hat{j_r}\in S$ for every $r$.
\end{itemize}
Since $OD(S)\leq \lambda-1$ then two of the indices $\hat{j_r}$ must lie in the same parity block of $S$. We may assume that they correspond to two vertices $v_{j_r}$ and $v_{j_{r+1}}$. Since $\hat{j_r}$ and $\hat{j_{r+1}}$ lie in the same parity interval there are an even number of vertices of $T\cup K'$ between $v_{j_r}$ and $v_{j_{r+1}}$. Therefore, according to \eqref{thRadCyc} they get different signs in the Radon partition.
	
Using Proposition~\ref{propRadTr} we have that $\aff(T)$ does not intersect $\conv(K')$. Thus $\aff(T)$ does not intersect $\conv(K)$ and therefore it cannot be a transversal.
\end{proof}

Combining this proposition with the combinatorial lemmas of Subection~\ref{secTools} we may now prove the upper bound of $\zeta(k,d,\lambda)$, that is
$$
\zeta(k,d,\lambda)\leq (d-\lambda+1) + \flo{(2-\alpha(d,\lambda))(k-1)}.
$$

\begin{proof}[Proof of the upper bound Theorem~\ref{thCyclic}]
Suppose $T$ is a transversal. By the second part of Proposition~\ref{propInter}, the sets $I(d,\lambda,k)$ of Section~\ref{secTools} satisfy the hypothesis for $S$ in Proposition~\ref{propIneq}. Adding the corresponding inequalities for each of the $2\cei{\frac{d}{2}}-\lambda+1$ sets we get:
$$
\sum_{j}\sum_{\alpha\in I(d,\lambda,j)} |A_\alpha| \leq \left(2\cei{\frac{d}{2}}-\lambda+1\right)(k-1).
$$
Using the first part of Proposition~\ref{propInter}, we have that each $|A_\alpha|$ appears exactly $\cei{\frac{d}{2}}$ times in the left hand side sum. Therefore,
$$
\cei{\frac{d}{2}}\cdot\sum_{\alpha=1}^{d-\lambda+2} |A_\alpha| \leq \left(2\cei{\frac{d}{2}}-\lambda+1\right)(k-1).
$$
Finally, dividing by $\cei{\frac{d}{2}}$ and adding the points from the transversal we get
$$
|V| = (d-\lambda+1) + \sum_{\alpha=1}^{d-\lambda+2} |A_\alpha| \leq (d-\lambda+1) + \left(\frac{2\cei{\frac{d}{2}}-\lambda+1}{\cei{\frac{d}{2}}}\right)(k-1).
$$
Since $|V|$ is an integer, we can take the floor function on both sides to get the desired result.
\end{proof}

Therefore, we obtain as a corollary the following upper bound for $m^*(k,d,\lambda)$.

\begin{corollary}
In the non-trivial range, when $\alpha(d,\lambda)< 1$ we have that
$$
m^*(k,d,\lambda)\le (d-\lambda+1) + \flo{(2-\alpha(d,\lambda))(k-1)} = Z(k,d,\lambda).
$$
\end{corollary}

\subsection{Lower bound of Theorem~\ref{thCyclic}}\label{secLow}

Now we want to show that the cyclic polytope with $z(k,d,\lambda)$ points always has a complete Kneser $(d-\lambda)$-transversal to the convex hulls of all its $k$-sets. We will use the notation from Section~\ref{secUp}. Notice that to determine which points are the $d-\lambda+1$ that will generate the transversal, it is enough to say how many points are in each set $A_i$.

The idea for constructing the example and proposing a complete Kneser transversal will be to fix the parameters $k, d, \lambda$ and then distribute the points as evenly as possible in the sets $A_i$. However, it might be the case that distributing the points using \emph{all} the sets $A_i$ is not optimal. Sometimes we get a better example by evenly distributing the points only in some of the first $A_i$'s. This is the role that $j$ plays in the following formula:
$$
z(k,d,\lambda) = (d-\lambda+1) + \max_{\substack{j\in \{\lambda+1,\ldots,d-\lambda+2\}\\ j+\lambda \text{ is odd}}} \left(\flo{\frac{k-1}{\beta(\lambda,j)}}\cdot j + (k-1)_{\bmod{\beta(\lambda,j)}}\right).
$$
Remember that we had defined $\beta(\lambda,j)=\frac{j+\lambda-1}{2}$ when $j+\lambda$ is odd. We can extend this definition for both parities of $j+\lambda$ by defining $\beta(\lambda,j)$ as $\cei{\frac{j+\lambda-1}{2}}$.

We give a brief explanation of the formula to make the proof clearer. The $d-\lambda+1$ in the formula represents the points that generate the complete Kneser transversal. For a fixed value of $j$, we will distribute
$$
z_j := \flo{\frac{k-1}{\beta(\lambda,j)}}\cdot j + {(k-1)}_{\bmod{\beta(\lambda,j)}}.
$$
points evenly in the sets $A_1$, $\ldots$, $A_j$ and we will show that we indeed get a complete Kneser transversal. It may seem strange that in the index of the $\max$ in the formula for $z(k,d,\lambda)$ we have the restrictions $j\geq \lambda+1$ and $j+\lambda$ is odd. The first one is due to Proposition~\ref{propIneq} and the second one is just a refinement to improve optimality (as we will see in the proof). We can also create examples when $j+\lambda$ is even, but it turns out that they are not optimal. We work having in mind that we want to stretch the idea of even distribution as much as possible.

So, let us fix a value of $j$. We need the following auxiliary result.

\begin{proposition}\label{propOpti}
The solution to the optimization problem in variables $a$, $r$, maximize the function $aj+r$ subject to the restrictions
\begin{enumerate}
\item\label{r1}
$a$ is a non-negative integer,
\item\label{r2}
$r$ is in $\{0,1,\ldots,j-1\}$,
\item\label{r3}
If $r\geq \beta(\lambda,j)$, then $\beta(\lambda,j) (a+1) \leq k-1$,		
\item\label{r4}
If $r\leq \beta(\lambda,j)-1$, then $(a+1)r + (\beta(\lambda,j)-r) a \leq k-1$.
\end{enumerate}
is $z_j$ and is obtained when
$$
a=\flo{\frac{k-1}{\beta(\lambda,j)}} \quad\text{ and }\quad r={(k-1)}_{\bmod{\beta(\lambda,j)}}.
$$
\end{proposition}

\begin{proof}
It is clear that to maximize the expression $aj+r$, first we have to maximize $a$ given our constraints, and then maximize $r$.

We will first get the maximum when $r\geq \beta(\lambda,j)$. In this case, we are subject to the restriction~\eqref{r3}:
$$
\beta(\lambda,j) (a+1) \leq k-1.
$$
The best value we can get for $a$ is $\flo{\frac{k-1}{\beta(\lambda,j)}}-1$. Since in this case we only have restriction~\eqref{r2} for $r$, then we can get $r=j-1$. Therefore, in this case the maximum for $aj+r$ is
$$
\left(\flo{\frac{k-1}{\beta(\lambda,j)}}-1\right)\cdot j + (j-1) = \flo{\frac{k-1}{\beta(\lambda,j)}} \cdot j -1.
$$
In the case $r\leq \beta(\lambda,j)-1$, we have restriction~\eqref{r4}, which can be simplified to
$$
\beta(\lambda,j) a + r \leq k-1.
$$
Once again, we first optimize $a$. The greatest value it can take is $\flo{\frac{k-1}{\beta(\lambda,j)}}$. Afterwards, by restriction~\eqref{r4} the maximum value that $r$ can take is ${(k-1)}_{\bmod{\beta(\lambda,j)}}$. Now we just check that this value of $r$ lies in the case we are studying:
$$
r = {(k-1)}_{\bmod{\beta(\lambda,j)}}\leq \beta(\lambda,j)-1.
$$
Therefore, when $r\leq \beta(\lambda,j)-1$ the best value for $aj+r$ is
$$
\flo{\frac{k-1}{\beta(\lambda,j)}}\cdot j + {(k-1)}_{\bmod{\beta(\lambda,j)}}.
$$
This value is always greater than the optimal value of the case $r\geq \beta(\lambda,j)$. Therefore, the solution for the whole optimization problem is $z_j$, it is attained in the case $r\leq \beta(\lambda,j)-1$ and the values of $a$ and $r$ are
$$
a = \flo{\frac{k-1}{\beta(\lambda,j)}} \quad\text{ and }\quad r = {(k-1)}_{\bmod{\beta(\lambda,j)}}.
$$
\end{proof}

The following proposition provides the connection between Proposition~\ref{propOpti} and the construction of examples with complete Kneser transversals.

\begin{proposition}\label{propConstruction}
Let $j$ be an integer in $\{\lambda+1,\ldots,d-\lambda+2\}$ and $a$ and $r$ integers that satisfy the constraints from Proposition~\ref{propOpti}. The cyclic polytope in dimension $d$ with
$$
(d-\lambda+1)+aj+r
$$
vertices has a complete Kneser $(d-\lambda)$-transversal to the convex hulls of all its $k$-sets.
\end{proposition}

\begin{proof}
We will use the notation from Section~\ref{secUp}. As we have said before, we will just split the $aj+r$ points among $A_1$, $\ldots$, $A_j$. For this, we choose $T=\{w_1, \ldots, w_{d-\lambda+1}\}$ in such a way that
$$
|A_1|=|A_2|=\ldots=|A_r|=a+1,\quad |A_{r+1}|=|A_{r+2}|=\ldots=|A_j|=a.
$$
First, we indeed have that $|\mathcal{A}|=r(a+1)+(j-r)a=ja+r$. We claim that $\aff(T)$ is indeed the desired transversal. Consider a $k$-set $K$ of $\mathcal{A}$. If $K$ has an element $w_i$, then clearly $\aff(T)\cap\conv(K)\neq\emptyset$. If $K\cap T=\emptyset$, the constraints (3) and (4) from Proposition~\ref{propOpti} guarantee that $K$ cannot be contained in the union of $\beta(\lambda,j)$ sets $A_i$. This means that if $I$ is the smallest family of indices such that
$$
K\subseteq \bigcup_{i\in I} A_i,
$$
then $|I|\geq \beta(\lambda,j)+1 = \cei{\frac{j+\lambda+1}{2}}\geq \frac{j+\lambda+1}{2}$. But then
$$
2|I|-\lambda-1\geq (j+\lambda+1)-\lambda-1= j,
$$
and therefore we can use Proposition~\ref{propAlter} to conclude that $I$ contains a set with at least $\lambda$ odd changes, and therefore that it has at least $\lambda+1$ parity blocks. Then, we can find a set $J\subseteq I$ with $\lambda+1$ elements of alternating parity. For each $j\in J$ choose a vertex $v_j$ in $K\cap A_j$. Call $K'$ the set of all these $v_j$'s.
	
If $v_{j_1}$ and $v_{j_2}$ are consecutive vertices in $K'$, then there is an odd number of elements from $T\cup K'$ in between them. This means that all the elements from $K'$ get the same sign in the Radon partition for $T\cup K'$. Therefore, by Proposition~\ref{propMinTr} and Proposition~\ref{propRadTr} we have that $\aff(T)\cap\conv(K)\neq\emptyset$. We have proven that $\aff(T)$ is the desired transversal.
\end{proof}

We may now prove the lower bound of $\zeta(k,d,\lambda)$, that is
$$
\zeta(k,d,\lambda)\geq (d-\lambda+1) + \max_{\substack{j\in \{\lambda+1,\ldots,d-\lambda+2\}\\ j+\lambda \text{ is odd}}} \left(\flo{\frac{k-1}{\beta(\lambda,j)}}\cdot j + {(k-1)}_{\bmod{\beta(\lambda,j)}}\right).
$$

\begin{proof}[Proof of the lower bound Theorem~\ref{thCyclic}]
We combine Proposition~\ref{propConstruction} and the optimal value from Proposition~\ref{propOpti} for every integer $j$ in $\{\lambda+1,\ldots,d-\lambda+2\}$. We can show a corresponding complete Kneser $(d-\lambda)$-transversal for the cyclic polytope with $(d-\lambda+1)+z_j$ vertices.
	
Notice that if $j+\lambda$ is even, then we would already have an example with the same number of points by using $j-1$ instead of $j$. Therefore, we do not lose examples by requiring that $j+\lambda$ be odd. By definition of $\zeta$ we can take the maximum over all these values of $j$ and thus this yields the claimed bound.
\end{proof}

\begin{remark}
The above proof is the best we can do by using a balanced distribution of points in the sets $A_i$. Let us briefly explain this by sketching an argument. Suppose that we distribute the points in the sets $A_{i_1}$, $A_{i_2}$, $\ldots$, $A_{i_j}$, where $i_1<\ldots<i_j$. We have that the set $\{i_1,\ldots,i_j\}$ has at most $j$ parity blocks, not more than those in the set $\{1,2,\ldots,j\}$ used in Proposition~\ref{propConstruction}. So when we get to an optimization problem analogous to the one in Proposition~\ref{propOpti}, we get a problem with stronger constraints, and this translates to an example with fewer or the same number of points. Nevertheless, as far as we know, there might be still better examples if we do not require an even distribution.
\end{remark}

\subsection{Asymptotics}

The bounds found by Theorem~\ref{thCyclic} are asymptotically correct in terms of~$k$. 

\begin{proof}[Proof of Theorem~\ref{asymp1}]
By Theorem~\ref{thCyclic} we have that 
$$
\frac{\zeta(k,d,\lambda)}{k}\le \frac{Z(k,d,\lambda)}{k}\le  \frac{(2-\alpha(d,\lambda))(k-1)}{k}+\frac{d-\lambda+1}{k},
$$
hence 
$$
\lim_{k\to\infty}\frac{\zeta(k,d,\lambda)}{k}\le 2-\alpha(d,\lambda).
$$
We will prove now that $\lim_{k\to\infty}\frac{\zeta(k,d,\lambda)}{k}\ge 2-\alpha(d,\lambda)$.

By the lower bound of Theorem~\ref{thCyclic} we have that 
$$
\zeta(k,d,\lambda)\ge z(k,d,\lambda)\ge \flo{\frac{k-1}{\beta(\lambda,j)}}\cdot j + {(k-1)}_{\bmod{\beta(\lambda,j)}}
$$
for every $j\in \{\lambda+1,\ldots,d-\lambda+2\}$ whenever $d+\lambda$ is odd. Suppose $d$ is odd. Considering $j=d-\lambda +2$ yields
$$
\zeta(k,d,\lambda)\ge \flo{\frac{k-1}{\frac{d+1}{2}}}(d-\lambda +2)+(d-\lambda+1).
$$
Then 
$$
\lim_{k\to\infty}\frac{\zeta(k,d,\lambda)}{k}\ge 2\left(\frac{d-\lambda+2}{d+1}\right)=2\left(1-\frac{\lambda-1}{d+1}\right)=\left(2-\frac{\lambda-1}{\frac{d+1}{2}}\right).
$$
Since $d$ is odd, the rightmost expression equals $2-\alpha(d,\lambda)$, as desired.

Suppose now that $d$ is even. Then $(d-\lambda+1)+\lambda$ is odd and we may consider $j=d-\lambda +1$ yielding that
$$
\zeta(k,d,\lambda)\ge \flo{\frac{k-1}{\frac{d}{2}}}(d-\lambda +1)+(d-\lambda+1).
$$
Then $\lim_{k\to\infty}\frac{\zeta(k,d,\lambda)}{k}\ge 2(\frac{d-\lambda+1}{d})=2-\alpha(d,\lambda)$ and once again we obtain the correct lower bound.
\end{proof}

\begin{remark}\label{rem:mm}
As a consequence of  \cite[Corollary 5.1]{BM}, it is easy to prove that
$$
2-\frac{1}{d}\leq \lim_{k\to\infty}\frac{m(k,d,2)}{k},
$$
and therefore that for $d\geq3$ and $k$ large enough, Conjecture~1 of \cite{ABMR} is false. Moreover,
$$
\lim_{k\to\infty}\frac{m^*(k,d,2)}{k}\leq \lim_{k\to\infty}\frac{\zeta(k,d,2)}{k}=2-\frac{2}{d}<2-\frac{1}{d}\leq \lim_{k\to\infty}\frac{m(k,d,2)}{k}.
$$
So, for $k$ large enough and $d\geq 3$, $m^*(k,d,2)<m(k,d,2)$.
\end{remark}

\begin{remark}
In the non-trivial range, when $k>\lambda +1$ and $k-1\geq \left \lceil{\frac{d}{2}}\right \rceil$, we have
$$
\zeta(k,d,\lambda) = \flo{(2-\alpha(d,\lambda))(k-1)}	+ (d-\lambda)+1 < d+2(k-\lambda).
$$
\end{remark}

\section{Some exact values of $m^*$}\label{exact}

If we set additional constraints in $k$, $d$ and $\lambda$ then we can find the exact value of $m^*(k,d,\lambda)$. An easy example is the case $k=\lambda$.

\begin{proposition}\label{propEasy}
The value of $m^*(k,d,k)$ is $d$.
\end{proposition}

\begin{proof}
If we have $d$ points or less in $\mathbb{R}^d$, then we can choose any subset $T$ with $d-k+1$ elements, and it will have non-empty intersection with any $k$-set. Therefore, $\text{aff}(T)$ will be a Kneser transversal. On the other hand, if we choose $d+1$ affinely independent points in $\mathbb{R}^d$, then any $(d-k+1)$-set $T$ will leave $k$ points in its complement, and therefore by affine independence $\text{aff}(T)$ cannot be a Kneser transversal.
\end{proof}

Here we present other cases in which we can obtain the exact value of $m^*(k,d,\lambda)$.
  
\begin{proof}[Proof of Theorem~\ref{exactvaluesm^*}]
By Theorems~\ref{lowerbound} and \ref{thCyclic}, we have that 
$$
(d-\lambda+1)+k\le m^*(k,d,\lambda)\le(d-\lambda+1) + \flo{(2-\alpha(d,\lambda))(k-1)},
$$
when $\alpha(d,\lambda)<1$ and $k\ge \cei{\frac{d}{2}} + 1$. Then $m^*(k,d,\lambda)=(d-\lambda+1)+k$ whenever $k = \flo{(2-\alpha(d,\lambda))(k-1)}$. Since for any real number $x$ we have that $\flo{x}\le x<\flo{x} +1$, then it follows that $m^*(k,d,\lambda)=(d-\lambda+1)+k$ if
$$
k\le (2-\alpha(d,\lambda))(k-1) < k+1.
$$
Doing some calculations, we have the following chain of equivalences
\begin{align*}
\frac{k}{k-1}&\le 2-\alpha(d,\lambda)<\frac{k+1}{k-1}, \\[1.5ex]
1+ \frac{1}{k-1}&\le 2-\alpha(d,\lambda)<1+\frac{2}{k-1}, \\[1.5ex]
\frac{1}{k-1}&\le 1-\alpha(d,\lambda)<\frac{2}{k-1}, \\[1.5ex]
k-1&\ge \frac{1}{1-\alpha(d,\lambda)}>\frac{k-1}{2},
\end{align*}
concluding that $m^*(k,d,\lambda)=(d-\lambda+1)+k$ whenever $\frac{2}{1-\alpha(d,\lambda)}+1>k\geq \frac{1}{1-\alpha(d,\lambda)}+1$.

The inequality $\frac{2}{1-\alpha(d,\lambda)}+1>k$ holds by hypothesis. On the other hand as $\alpha(d,\lambda)<1$ we have that $\lambda\le\cei{\frac{d}{2}}$. Then the inequality $k\ge\frac{1}{1-\alpha(d,\lambda)}+1$ holds since by hypothesis we have that $k\ge\cei{\frac{d}{2}}+1$. Therefore $m^*(k,d,\lambda)=(d-\lambda+1)+k$.
\end{proof}

When $\lambda=1$, we have obtained the exact value of $m^*(k,d,1)$.

\begin{proof}[Proof of Theorem~\ref{m^*(k,d,1}]
The theorem is trivial for $k=1$, in the other cases the upper bound is given by Theorem~\ref{thCyclic}. Let $X$ be any set of $d+2(k-1)$ points in $\R^d$, in order to prove that $d+2(k-1)\le m^*(k,d,1)$, we will show that there exists an hyperplane $H$ with at least $d$ points of $X$ and with at most $k-1$ points of $X$ in each of the two open half-spaces determined by $H$. Let $F$ be a face of $\conv(X)$ and let us consider $S\subset X\cap F$ with cardinality $d-1$, such that each point of $S$ is a vertex of the $d$-polytope $\conv(X)$. Consider the hyperplane $\aff(F)$, clearly $X$ is contained in one half-space of $\aff(F)$. We may now rotate continuously $\aff(F)$ by fixing $S$ to find an hyperplane $H$ with at least $d$ points of $X$ and with at most $k-1$ points  in each of the two open half-spaces determined by $H$. Then $H$ is a complete Kneser $d$-transversal to the convex hulls of all $k$-sets.
\end{proof}

By Theorem~\ref{cyceasy} we immediately have 

\begin{corollary}\label{igualdadm^*}
$m^*(k,d,\lambda)=d-\lambda+k$ when $\alpha(d,\lambda)\geq 1$.
\end{corollary}

We can now give a strict inequality between $m^*$ and $m$.

\begin{corollary}\label{m^*<m}
$m^*(k,d,\lambda)<m(k,d,\lambda)$ when $k>\lambda$ and  $\alpha(d,\lambda)\geq 1$.
\end{corollary}

\begin{proof}
Since $d-\lambda+k+\cei{\frac{k}{\lambda}}-1\le m(k,d,\lambda)$ (\cite[Corollary 1]{ABMR}) and $m^*(k,d,\lambda)=d-\lambda+k$ (Corollary~\ref{igualdadm^*}), it follows that $m^*(k,d,\lambda)<d-\lambda+k+\cei{\frac{k}{\lambda}}-1\le m(k,d,\lambda)$ if $1<\cei{\frac{k}{\lambda}}$. The result holds since by hypothesis $k>\lambda$.
\end{proof}

\begin{remark}
In the trivial range $\alpha(d,\lambda)\geq 1$, there exist configurations of $d-\lambda+k+\cei{\frac{k}{\lambda}}-1$ points in $\R^d$ with a $(d-\lambda)$-transversal to the convex hulls of all the $k$-sets and without complete Kneser $(d-\lambda)$-transversals.
For infinitely values of $d,\lambda$ and $k$, these configurations of points are the cyclic polytopes (see Section~\ref{cyclic}). For instance, the cyclic polytope with $6$ points in general position in $\R^4$ has transversal lines to the convex hulls of all the $4$-sets but does not have complete Kneser transversal lines. This means $m^*(4,4,3)=5$ and $6\le m(4,4,3)$.
\end{remark}

We finish this work by mentioning that the smallest values of $k,d$ and $\lambda$ for which we do not know the value of $m^*(k,d,\lambda)$ are $k=3$, $d=5$ and $\lambda=2$. We only have the bounds 
$$
7\le m^*(3,5,2)\le 8.
$$

\section*{Acknowledgements}

The authors would like to thank the editor and the anonymous referee for their valuable comments and suggestions.



\begin{thebibliography}{}

\bibitem{ABMR}
J.L.~Arocha, J.~Bracho, L.~Montejano, J.L.~Ram\'{i}rez Alfons\'{i}n,
Transversals to the convex hulls of all $k$-sets of discrete subsets of $\mathbb{R}^{n}$,
J. Combin. Theory Ser. A  118 (2011), 197--207.

\bibitem{BL}
A.~Bj\"{o}rner, M.~Las Vergnas, B.~Sturmfels, N.~White, G.M.~Ziegler,
Oriented Matroids, Cambridge University Press, Cambridge, 1999.

\bibitem{BM}
B.~Bukh, J.~Matou\v sek, G.~Nivasch,
Stabbing simplices by points and flats,
Discrete Comput. Geom. 43 (2010), 321--338.

\bibitem{kneser}
M.~Kneser,
Aufgabe 360,
Jahresber. Dtsch. Math.-Ver. 58 (1955), 27.

\bibitem{lovasz}
L.~Lov\'{a}sz,
Kneser's conjecture, chromatic number and homotopy,
J. Combin. Theory Ser. A 25 (1978), 319--324.

\bibitem{Neumann}
B.H.~Neumann,
On an invariant of plane regions and mass distributions,
J. London Math. Soc. 20 (1945), 226--237.

\bibitem{Rado}
R.~Rado,
A theorem on general measure,
J. London Math. Soc. 21 (1947), 291--300.

\end{thebibliography}
\end{document}